\documentclass[11pt]{article}

\usepackage{amsmath} 
\usepackage[mathscr]{eucal}
\usepackage{amssymb} 
\usepackage{theorem}
\usepackage{tikz-cd}
\usetikzlibrary{positioning}

\usepackage{enumerate}

\theoremstyle{change}  
\newtheorem{theorem}{Theorem}[section] 
\newtheorem{lemma}[theorem]{Lemma}  
\newtheorem{proposition}[theorem]{Proposition}
\newtheorem{corollary}[theorem]{Corollary}
\theorembodyfont{\rmfamily}  
\newtheorem{Remark}[theorem]{Remark}

\newtheorem{definition}[theorem]{Definition}
\newtheorem{notation}[theorem]{Notation}
\newtheorem{nothing}[theorem]{} 

\newenvironment{proof}{\noindent{\bf Proof}\ }{\qed\bigskip}

\renewcommand{\le}{\leqslant} 
\renewcommand{\ge}{\geqslant}

\newcommand{\Aut}{\mathrm{Aut}}

\newcommand{\calC}{\mathcal{C}}

\newcommand{\calE}{\mathcal{E}}
\newcommand{\calF}{\mathcal{F}}

\newcommand{\calO}{\mathcal{O}}

\newcommand{\catfont}{\mathsf}

\newcommand{\CC}{\mathbb{C}}

\newcommand{\Def}{\mathrm{Def}}

\newcommand{\End}{\mathrm{End}}

\newcommand{\FF}{\mathbb{F}}

\newcommand{\HsetGD}{\lset{H}_G^{\Delta}}
\newcommand{\KsetHD}{\lset{K}_H^{\Delta}}
\newcommand{\KsetGD}{\lset{K}_G^{\Delta}}
\newcommand{\Hom}{\mathrm{Hom}}

\newcommand{\Ind}{\mathrm{Ind}}
\newcommand{\ind}{\mathrm{ind}}

\newcommand{\Inf}{\mathrm{Inf}}

\newcommand{\Isom}{\mathrm{Iso}}

\newcommand{\KK}{\mathbb{K}}

\newcommand{\lexp}[2]{\setbox0=\hbox{$#2$} \setbox1=\vbox to
                 \ht0{}\,\box1^{#1}\!#2}

\newcommand{\lMod}[1]{\llap{\phantom{|}}_{#1}\catfont{Mod}}

\newcommand{\lset}[1]{\llap{\phantom{|}}_{#1}\catfont{set}}

\newcommand{\NN}{\mathbb{N}}

\newcommand{\Out}{\mathrm{Out}}

\newcommand{\qed}{\nobreak\hfill
                  \vbox{\hrule\hbox{\vrule\hbox to 5pt
                  {\vbox to 8pt{\vfil}\hfil}\vrule}\hrule}}
\newcommand{\Res}{\mathrm{Res}}
\newcommand{\res}{\mathrm{res}}

\newcommand{\ZZ}{\mathbb{Z}}

\newcommand{\BD}{B^{\Delta}}

\newcommand{\be}{\textbf{e}^\Delta}

\newcommand{\FFTD}{\mathbb{F}T^{\Delta}}

\newcommand{\Proj}{\mathrm{Proj}}

\renewcommand{\beta}{\delta}

\title{$p$-Bifree biset functors}
\author{Olcay COŞKUN$^1$ and Deniz YILMAZ$^2$}

\date{
	$^1$ Mathematics Research Center, Baku, Azerbaijan \\ 
   \texttt{olcay.coshkun@bsu.edu.az}\\%
	$^2$ Bilkent University, Ankara, Turkey \\ \texttt{d.yilmaz@bilkent.edu.tr} \\[2ex]%
}

\begin{document}
\sloppy

\maketitle
\begin{abstract}
We introduce and study the category of $p$-bifree biset functors for a fixed prime $p$, defined via bisets whose left and right stabilizers are $p'$-groups. This category naturally lies between the classical biset functors and the diagonal $p$-permutation functors, serving as a bridge between them. Every biset functor and every diagonal $p$-permutation functor restricts to a $p$-bifree biset functor.

We classify the simple $p$-bifree biset functors over a field $K$ of characteristic zero, showing that they are parametrized by pairs $(G,V)$, where $G$ is a finite group and $V$ is a simple $K\mathrm{Out}(G)$-module. As key examples, we compute the composition factors of several representation-theoretic functors in the $p$-bifree setting, including the Burnside ring functor, the $p$-bifree Burnside functor, the Brauer character ring functor, and the ordinary character ring functor.
We further investigate classical simple biset functors, $S_{C_p \times C_p, \mathbb{C}}$ and $S_{C_q \times C_q, \mathbb{C}}$ for a prime $q\neq p$. 

\end{abstract}

\section{Introduction}

The theory of biset functors, which was introduced and extensively developed by Bouc, occupies a central position in the functorial representation theory of finite groups. It enables a unified treatment of representation rings when the structural maps restriction, induction, deflation, inflation, and isogation are present. The completion of the classification of endo-permutation modules of $p$-groups \cite{Bouc06} and the description of the unit group of Burnside rings of $p$-groups \cite{Bouc07}, both due to Bouc, are two notable applications of the theory of biset functors.

Diagonal $p$-permutation functors, introduced by Bouc and the second author \cite{BoucYilmaz2020}, provide a functorial framework for studying structures involving actions of $p$-permutation bimodules with additional constraints. By replacing bisets with $p$-permutation bimodules whose vertices are twisted diagonals, this theory captures essential representation-theoretic phenomena, particularly those related to block theory. Diagonal $p$-permutation functors have already found applications in the block theory of finite groups; see for instance the finiteness result in terms of functorial equivalences, Theorem~10.6 in \cite{BoucYilmaz2022}, which is in the spirit of Puig's and Donovan's finiteness conjectures.

Although the two theories of biset functors and diagonal $p$-permutation functors are both defined on categories of finite groups, there is no direct functorial connection between them. However, their morphisms are related as follows. A diagonal $p$-permutation bimodule is a $p$-permutation bimodule whose indecomposable summands have (twisted) diagonal subgroups as vertices. The linearization map applied to bisets yields permutation, and hence $p$-permutation bimodules. In particular, the elementary bisets
\[
\Res^G_H\;\ (H \le G),\;\
\Ind^G_H\;\  (H \le G), \;\
\Isom(f) \;\  (f:G \xrightarrow{\sim} H),\]\[ 
\Inf^G_{G/N}, \Def^G_{G/N} \;\  (N \unlhd G,\ N \text{ is } p'\text{-normal})
\]
give rise to morphisms in the diagonal $p$-permutation category via linearization. In contrast, inflations and deflations along general normal subgroups do not yield diagonal $p$-permutation bimodules under linearization; see \cite[Lemma~4.2]{BoucYilmaz2020}. This obstruction prevents the existence of a direct functor between the categories of biset functors and of diagonal $p$-permutation functors.

Motivated by this observation, we introduce the notions of $p$-bifree bisets and $p$-bifree biset functors. A \emph{$p$-bifree biset} is a biset with $p'$-stabilizers on both sides. The category of $p$-bifree bisets is similar to the classical biset category, but only includes inflations and deflations via $p'$-normal subgroups. For a commutative ring $R$ with unity, we denote by $R\mathcal{C}^\Delta$ the category whose objects are finite groups and whose morphisms are given by the $R$-linear extension of the Grothendieck group $RB^\Delta(H,G)$ of $p$-bifree $(H,G)$-bisets. An $R$-linear functor from $R\mathcal{C}^\Delta$ to the category of $R$-modules is called a \emph{$p$-bifree biset functor over $R$}. Note that via restriction and linearization, every biset functor and every diagonal $p$-permutation functor becomes a $p$-bifree biset functor.

With this definition, the category of $p$-bifree biset functors lies naturally between classical biset functors and diagonal $p$-permutation functors. On one hand, it contains all classical biset functors via restriction to the $p$-bifree part of the biset category. On the other hand, diagonal $p$-permutation functors factor through it via the linearization map, since only bisets with $p'$-stabilizers induce diagonal $p$-permutation bimodules under linearization. In this sense, $p$-bifree biset functors form a common generalization of both theories, effectively building a bridge between them. 

Using Bouc's theory on linear functors \cite{Bouc1996}, we show that the simple $p$-bifree biset functors $S^\Delta_{G,V}$ are parametrized by the pairs $(G,V)$ of a finite group $G$ and a simple $R\Out(G)$-module $V$.  We also consider various representation rings as $p$-bifree biset functors and determine their composition factors. Our main results can be summarized as follows. 

\begin{itemize}
\item Let $\KK$ be a field of characteristic zero. We prove that (Corollary~\ref{cor compfactorBurnside}) the composition factors of the Burnside ring as the $p$-bifree biset functor are exactly the functors $S^\Delta_{G,\KK}$ where $G$ is a $B^\Delta_p$-group. See Definition~\ref{def BDelta-group} for the definition of $B^\Delta_p$-group. The proof follows \cite[Chapter~5]{Bouc2010a} very closely. 
\item The composition factors of the representable $p$-bifree biset functor $\KK B^\Delta(-,1)$ are exactly the functors $S^\Delta_{G,\KK}$ where $G$ is a $B$-group of $p'$-order (Corollary~\ref{cor compfactorsrepresentable}). 
\item Let $G$ be a finite group. The $\KK$-dimension of $S^\Delta_{1,\KK}(G)$ is equal to the number of conjugacy classes of cyclic $p'$-subgroups of $G$ (Theorem~\ref{thm dimensionoftrivialsimple}).
\item We also investigate the $p'$-bifree biset functor structures of the Brauer and ordinary character rings of finite groups.  Let $k$ be an algebraically closed field of characteristic $p>0$ and let $R_k(G)$ and $R_\CC(G)$ denote the Grothendieck groups of $kG$ and $\CC G$-modules, respectively.  We show that both $R_k(-)$ and $R_\CC(-)$ are semisimple $p$-bifree biset functors. One has
    \begin{align*}
        \KK R_k\cong \bigoplus_{(m,\xi)}S^\Delta_{\ZZ/m\ZZ, \CC_\xi}
    \end{align*}
    where $(m,\xi)$ runs through a set of pairs consisting of a positive $p'$-number $m$ and a primitive character $\xi$ of $(\ZZ/m\ZZ)^\times$ (Corollary~\ref{cor Brauercharacter}). Furthermore,
       \begin{align*}
        \CC R_\CC\cong \bigoplus_{(m,\xi)} S^\Delta_{\ZZ/m\ZZ, \CC_\xi}
    \end{align*}
    where $(m,\xi)$ runs through the set of pairs consisting of a positive integer $m$ and a character $\xi$ of $(\ZZ/m\ZZ)^\times$ such that $\xi_{p'}$ is primitive (Corollary~\ref{cor complexcharacter}). As a result of these we obtain a short exact sequence
\begin{align*}
    0\longrightarrow \bigoplus_{(m,\xi)} S^\Delta_{\ZZ/m\ZZ,\xi}\longrightarrow \CC R_\CC \longrightarrow \CC R_k\longrightarrow 0
\end{align*}
of $p$-bifree biset functors, where $(m,\xi)$ runs through the set of pairs consisting of a positive integer $m$ divisible by $p$ and an irreducible character $\xi$ of $(\ZZ/m\ZZ)^\times$ with primitive $p'$-part.
    \item Finally, we consider some simple biset functors as $p$-bifree biset functors. Let $S_{G,V}$ denote the simple biset functor associated to the pair $(G,V)$ of a finite group $G$ and simple $R\Out(G)$-module $V$. Then one has isomorphisms
    \begin{align*}
        S_{C_p\times C_p,\CC}\cong \bigoplus_{P: \text{non-cyclic } p\text{-group}} S^\Delta_{P,\CC}
    \end{align*}
and
    \begin{align*}
        S_{C_q\times C_q,\CC}\cong \bigoplus_{\substack{C_q\times C_q\times P\\ P: \text{cyclic } p\text{-group}}} S_{C_q\times C_q\times P,\CC}^\Delta
    \end{align*}
    of $p$-bifree biset functors (Theorems~\ref{thm elemabelian} and \ref{thm elemqabelian}).
\end{itemize}

\section{Category of $p$-free bisets}
Let $p>0$ be a prime and let $G,H$ and $K$ denote finite groups. 

\begin{definition}
    Let $X$ be an $(H,G)$-biset. We say $X$ is a \textit{$p$-bifree $(H,G)$-biset} if the left and right stabilizers of the $H\times G$-orbits of $X$ are $p'$-groups. In other words, $X$ is a disjoint union of transitive $(H,G)$-bisets of the form $\left[(H\times G)/L\right]$ where $k_1(L)$ and $k_2(L)$ are $p'$-groups.
\end{definition}

\begin{nothing}

\smallskip
{\rm (a)} Let $\HsetGD$ denote the category of $p$-bifree $(H,G)$-bisets. Using the Mackey formula, one shows that the composition of bisets induces a bilinear map $$\KsetHD\times \HsetGD\to \KsetGD.$$

Let $\BD(H,G)$ denote the Grothendieck group of the category $\HsetGD$. The product of bisets induces a well-defined map
\begin{align}\label{eqn composition}
   -\cdot_H-:\BD(K,H)\times \BD(H,G)\to \BD(K,G)\,.
\end{align}
For a subgroup $L\le H\times G$, set $q(L)=L/(k_1(L)\times k_2(L))$. One has canonical isomorphisms
\begin{align*}
    p_1(L)/k_1(L)\cong q(L)\cong p_2(L)/k_2(L)\,.
\end{align*}
Recall from \cite{Bouc2010a} that every transitive $(H,G)$-biset can be written as a product of five elementary bisets:
\begin{align*}
\left[\frac{H\times G}{L}\right] = \Ind^H_{p_1(L)} \circ \Inf^{p_1(L)}_{p_1(L)/k_1(L)} \circ \Isom(f)\circ \Def^{p_2(L)}_{p_2(L)/k_2(L)} \circ \Res^G_{p_2(L)}\,,
\end{align*}
where $f:p_2(L)/k_2(L)\to p_1(L)/k_1(L)$ is the canonical isomorphism. One has $\Inf^G_{G/N}\in \BD(G, G/N)$ if and only if $N$ is a $p'$-subgroup of $G$. Similarly, $\Def^G_{G/N}\in \BD(G/N,G)$ if and only if $N$ is a $p'$-subgroup of $G$. Hence $B^\Delta(H, G)$ is generated by bisets of the above form for $k_1(L)$ and $k_2(L)$ $p'$-groups.

\smallskip
{\rm (b)} Let $R$ be a commutative ring with $1$. Let $R\calC^\Delta$ denote the following category:
\noindent\begin{itemize}
\item objects: finite groups,
\item $\Hom_{\calC^\Delta}(G,H)=R\otimes_{\ZZ} \BD(H,G)=R\BD(H,G)$,
\item composition is induced from the map in (\ref{eqn composition}),
\item $\mathrm{Id}_{G}=[\mathrm{Id}_G]=[G]$.
\end{itemize}

\smallskip
{\rm (c)} Let $G$ be a finite group. We set
\begin{align*}
    I_G=\sum_{|H|<|G|}R\BD(G,H)\circ R\BD(H,G)\,.
\end{align*}
Note that $I_G$ is an ideal of the endomorphism ring $\End_{R\calC^\Delta}(G)=R\BD(G,G)$. The quotient

\begin{align*}
\calE^\Delta(G)=R\BD(G,G) /I_G
\end{align*}
is called the {\em essential algebra} of $G$ in $R\calC^\Delta$.

A transitive $p$-bifree $(G,G)$-biset $(G\times G)/L$ has zero image in $\calE^\Delta(G)$ if and only if $|q(L)|<|G|$.  Hence as in the usual biset category we have an isomorphism
\begin{align*}
\calE^\Delta(G)\cong R\Out(G) \,.
\end{align*}
\end{nothing}

\section{$p$-bifree biset functors}

In this section, we classify the simple objects in the category $\calF^\Delta_R$ of $p$-bifree biset functors over a commutative ring $R$. Our construction follows the general theory of linear biset functors developed by Bouc in \cite{Bouc1996}. 

\begin{definition}
    An $R$-linear functor $R\calC^\Delta\to \lMod{R}$ is called a {\em $p$-bifree biset functor over $R$}.
\end{definition}

Together with natural transformations, $p$-bifree biset functors over $R$ form an abelian category which we denote by $\calF^\Delta_R$. 

By \cite{Bouc1996}, the simple $p$-bifree biset functors over $R$ are parametrized by the pairs $(G,V)$ where $G$ is a finite group and $V$ is a simple $R\Out(G)$-module. For a given pair $(G,V)$ we denote the corresponding simple $p$-bifree biset functor by $S^\Delta_{G,V}$.

\begin{Remark}
    {\rm (a)} Every biset functor becomes a $p$-bifree biset functor via the restriction along the inclusion $R\calC^\Delta \subseteq R\calC$, where $R\calC$ denotes the usual biset category over $R$.

    \smallskip
    {\rm (b)} Let $R^\Delta_{pp_k}$ denote the diagonal $p$-permutation category. The linearization map $X\mapsto kX$ induces a functor $\calC^\Delta \to R^\Delta_{pp_k}$. This makes every diagonal $p$-permutation functor over $R$ a $p$-bifree biset functor.
\end{Remark}

\begin{definition}
    Let $F$ be a $p$-bifree biset functor and let $G$ be a finite group. We define the submodules $\underline{F}(G)$ and $\mathcal{J}F(G)$ of $F(G)$ by
    \begin{align*}
    \underline{F}(G) = \bigcap_{\substack{|H|<|G|\\ \alpha\in RB^\Delta(H, G)}}\mathrm{Ker}(F(\alpha): F(G)\to F(H))
\end{align*}
and
\begin{align*}
    \mathcal JF(G) = \sum_{\substack{|H|<|G|\\ \alpha\in RB^\Delta(G, H)}}\mathrm{Im}(F(\alpha): F(H)\to F(G))\,.
\end{align*}
\end{definition}
Notice that both $\underline{F}(-)$ and $\mathcal{J}F(-)$ are subfunctors of $F$. The subfunctor $\underline{F}(-)$ is called the \textit{restriction kernel} of $F$. These restriction kernels serve as a tool for detecting minimal groups and composition factors. They will be used repeatedly in later sections to describe the subfunctor structure and to prove that certain families of simple functors exhaust all composition factors. See Appendix for more properties of the restriction kernels.

\section{Burnside ring as a $p$-bifree biset functors}

Let $R$ be a commutative ring with identity and let $\KK$ be a field of characteristic zero. The {\em Burnside functor} $RB:R\calC^\Delta\to \lMod{R}$ is defined as
\begin{itemize}
\item $G\mapsto RB(G):=R\otimes_{\ZZ} B(G)$.
\item $X\in R\BD(H,G) \mapsto \left(RB(X):RB(G)\to RB(H),\quad U\mapsto X\cdot_G U\right)$.
\end{itemize}
In this section we analyze the structure of $RB$. It turns out that the results are very similar to those in Chapter 5 of \cite{Bouc2010a}, where the structure of the Burnside functor is described in terms of $B$-groups. In our setting, the same arguments apply almost line by line once one restricts to inflations and deflations along $p'$-normal subgroups. We include the full details here for completeness and to set up the theory in the context of $p$-bifree biset functors. The notion of a $B^\Delta_p$-group plays the role of a $B$-group in this setting, and many structural results, including the parametrization of composition factors and the poset of subfunctors, carry over with suitable modifications. The first and key result is the following characterization of evaluations of subfunctors. We refer to \cite[Lemma~5.2.1]{Bouc2010a} for its proof.

\begin{lemma}
    Let $F$ be a $p$-bifree biset subfunctor of $RB$. Then for any finite group $G$, the $R$-module $F(G)$ is an ideal of $RB(G)$.
\end{lemma}
Note that $\KK B(G)$ is a split semisimple $\KK$-algebra. In particular, every ideal of $\KK B(G)$ is generated by a set of primitive idempotents.

The primitive idempotents $e^G_H$ of $\KK B(G)$ are indexed by the conjugacy classes $[s_G]$ of subgroups $H\le G$.  Let $\be_G$ denote the $p$-bifree biset subfunctor of $\KK B$ generated by $e^G_G\in \KK B(G)$. We note that this is in general smaller than the biset subfunctor ${\bf e}_G$ of $\KK B$ generated by $e_G^G$. Recall from \cite[Theorem~5.2.4]{Bouc2010a} that for a normal subgroup $N$ of $G$ we have
\begin{align*}
\Def^G_{G/N} e^G_G = m_{G,N} e^{G/N}_{G/N}\,,
\end{align*}
where the deflation number $m_{G, N}$ is given by
\begin{align*}
m_{G,N}=\frac{1}{|G|}\sum_{XN=G}|X|\mu(X,G)\,.
\end{align*}
\begin{proposition}
Let $G$ be a finite group. The following conditions are equivalent.

\smallskip
{\rm (i)} The evaluation of the subfunctor $\be_G$ vanishes on all groups of smaller order, that is, $\be_G(H) = \{0\}$ for all $|H| < |G|$.

\smallskip
{\rm (ii)} If $\be_G(H) \ne \{0\}$ for some finite group $H$, then $G$ is isomorphic to a quotient $Y/X$ of a section $(Y, X)$ of $H$ with $X$ a $p'$-group.

\smallskip
{\rm (iii)} The deflation number $m_{G,N}$ vanishes for every non-trivial normal $p'$-subgroup $N \unlhd G$.

\smallskip
{\rm (iv)} The deflation $\Def^G_{G/N} e^G_G$ is zero in $\KK B^\Delta(G/N)$ for all non-trivial normal $p'$-subgroups $N$ of $G$.
\end{proposition}

\begin{proof}
The equivalence {\rm (iii)} $\Longleftrightarrow$ {\rm (iv)} follows directly from Theorem~5.2.4 in \cite{Bouc2010a}. 
Suppose {\rm (i)} holds. Let $N \unlhd G$ be a non-trivial $p'$-subgroup. Then the image of $e^G_G$ under deflation lies in $\be_G(G/N) = \{0\}$, so $\Def^G_{G/N} e^G_G = 0$, and hence {\rm (iv)} holds.
Now suppose {\rm (iv)} holds and let $H$ be a finite group such that $\be_G(H) \ne \{0\}$. Then there exists $\varphi \in \KK \BD(H, G)$ such that $\varphi(e^G_G) \ne 0$.In particular there exist sections $(Y,X)$ of $H$ and $(T,S)$ of $G$, with $X$ and $S$ $p'$-groups, and an isomorphism $f: Y/X \to T/S$ such that
\[
\Ind^H_Y \Inf^Y_{Y/X} \Isom(f) \Def^T_{T/S} \Res^G_T (e^G_G) \ne 0.
\]
Since the restriction of $e^G_G$ to a proper subgroup is zero, we have $T=G$. Condition {\rm (iv)} further implies that $S=1$. This proves {\rm (ii)}.
Finally, {\rm (ii)} $\Rightarrow$ {\rm (i)} is immediate.
\end{proof}

\begin{definition}\label{def BDelta-group}
A finite group $G$ is called a {\em $B^\Delta_p$-group} if for any non-trivial normal $p'$-subgroup $N$ of $G$, the deflation number $m_{G,N}$ is equal to zero.
\end{definition}

\begin{Remark} When the condition on $N$ is removed, we recover Bouc's definition of a $B$-group. These two definitions can be compared as follows:

    \smallskip
    {\rm (i)} Every $B$-group is a $\BD_p$-group. 
    
    \smallskip
    {\rm (ii)} Every $p$-group is a $\BD_p$-group.

    \smallskip
    {\rm (iii)} A $p$-group which is also a $B$-group is necessarily elementary abelian, by 5.6.9 of \cite{Bouc2010a}. 

    \smallskip
    {\rm (iv)} A $p'$-group is a $\BD_p$-group if and only if it is a $B$-group. 
\end{Remark}

\begin{theorem}
Let $G$ and $H$ be finite groups. Then the following hold:

\smallskip
{\rm (i)} If $H$ is isomorphic to a quotient of $G$ by a $p'$-subgroup, then $\be_G \subseteq \be_H$.

\smallskip
{\rm (ii)} If $H$ is a $\BD_p$-group and $\be_G \subseteq \be_H$, then $H$ is isomorphic to a quotient of $G$ by a $p'$-subgroup.

\smallskip
{\rm (iii)} If $F$ is a subfunctor of $\KK B$ and $H$ is a minimal group of $F$, then $H$ is a $\BD_p$-group, $F(H) = \KK e^H_H$, and $\be_H \subseteq F$. In particular, $\be_H(H)=\KK e^H_H$ if $H$ is a $\BD_p$-group. 

\smallskip
{\rm (iv)} The minimal group $\beta(G)$ of $\be_G$ is uniquely determined up to isomorphism. One has $\be_G = \be_{\beta(G)}$, and $\beta(G)$ is isomorphic to a quotient of $G$ by a $p'$-normal subgroup. Moreover, for any $p'$-normal subgroup $N \unlhd G$ such that $G/N \cong \beta(G)$, one has $m_{G,N} \ne 0$.
\end{theorem}

\begin{proof}
(i) Let $N \unlhd G$ be a $p'$-subgroup such that $G/N \cong H$ via an isomorphism $f$. Then
\[
e^G_G \Inf^G_{G/N} \Isom(f) e^H_H = e^G_G,
\]
so $e^G_G \in \be_H(G)$ and hence $\be_G \subseteq \be_H$.

\smallskip
(ii) Since $e^G_G \in \be_H(G) = \KK \BD(G,H)e^H_H$, there exist sections $(T,S)$ of $G$ and $(Y,X)$ of $H$ with $S$ and $X$ $p'$-groups and an isomorphism $f:Y/X \to T/S$ such that
\[
e^G_G \Ind^G_T \Inf^T_{T/S} \Isom(f) \Def^Y_{Y/X} \Res^H_Y e^H_H \ne 0.
\]
As in previous arguments, since the restriction of $e^H_H$ to a proper subgroup is zero, we have $Y = H$, and since $H$ is a $\BD_p$-group, we get $X = 1$. Also because $e^G_G[G/L]=0$ for any proper subgroup $L$, we have $T=G$ and the result follows. 

\smallskip
(iii) This follows from the proof of Proposition~5.4.9 in \cite{Bouc2010a} applied in the $p$-bifree setting. The minimality of $H$ implies that $F(H)$ is generated by $e^H_H$, and hence $\be_H \subseteq F$.

\smallskip
(iv) This is a direct adaptation of Proposition~5.4.10 in \cite{Bouc2010a}, replacing all normal subgroups with $p'$-normal subgroups. 
\end{proof}

The following theorem collects properties of the minimal group $\beta(G)$ of the subfunctor $\be_G$. As in the classical case, the key results follow from Theorem~5.4.10 in \cite{Bouc2010a}, and the same arguments apply here with the necessary modifications to the $p'$-normal setting.

\begin{theorem}
Let $G$ be a finite group.
\begin{enumerate}[(i)]
\item Let $H$ be a $\BD_p$-group that is isomorphic to a quotient of $G$ by a $p'$-normal subgroup. Then $H$ is also isomorphic to a quotient of $\beta(G)$ by a $p'$-normal subgroup.
\item Let $N \unlhd G$ be a $p'$-normal subgroup. The following are equivalent:
\begin{enumerate}[(a)]
    \item $m_{G,N} \ne 0$,
    \item $\beta(G)$ is isomorphic to a quotient of $G/N$ by a $p'$-subgroup,
    \item $\beta(G) \cong \beta(G/N)$.
\end{enumerate}
\item In particular, if $N \unlhd G$ is a normal $p'$-subgroup, then $G/N \cong \beta(G)$ if and only if $G/N$ is a $\BD_p$-group and $m_{G,N} \ne 0$.
\end{enumerate}
\end{theorem}

Let $\BD_p$-gr denote the class of $\BD_p$-groups, and let $[\BD_p\text{-gr}]$ be a fixed set of representatives of their isomorphism classes. Define a relation $\gg$ on $\BD_p$-gr by declaring $G \gg H$ if and only if $H$ is isomorphic to a quotient of $G$ by a $p'$-group. A subset $\mathcal{M} \subseteq \BD_p\text{-gr}$ is said to be \emph{closed} if, whenever $H \in \mathcal{M}$ and $G \gg H$, it follows that $G \in \mathcal{M}$. 

The following result describes the lattice of subfunctors of the Burnside functor $\KK B$ in terms of these closed subsets. Its proof is the same as that of Theorem~5.4.14 in \cite{Bouc2010a}, with the necessary modifications to the $p$-bifree bisets.

\begin{theorem}\label{thm subfunctorlattice}
Let $\mathcal{S}$ be the set of of $p$-bifree biset subfunctors of the Burnside functor $\KK B$, ordered by inclusion, and let $\mathcal{T}$ be the set of closed subsets of the set $[\BD_p\text{-gr}]$ of isomorphism classes of $\BD_p$-groups, ordered by inclusion.
Then the assignment
\[
\Theta: F \mapsto \{ H \in [\BD_p\text{-gr}] \mid \be_H \subseteq F \}
\]
defines an isomorphism of posets $\Theta: \mathcal{S} \xrightarrow{\sim} \mathcal{T}$.

Its inverse is given by
\[
\Psi: A \mapsto \sum_{H \in A} \be_H.
\]
\end{theorem}


We now describe the composition factors of the Burnside functor $\KK B$ as a $p$-bifree biset functor. For each $\BD_p$-group $G$, the subfunctor $\be_G$ is crucial, and its maximal proper subfunctor determines a unique simple quotient. We refer to Section 5 of \cite{Bouc2010a} for the proofs of the following results up to the end of the section.
\begin{proposition}
\label{prop simplequotient}
{\rm (i)} Let $G$ be a $\BD_p$-group. Then the functor $\be_G$ has a unique maximal subfunctor
\[
\textbf{j}_G = \sum_{\substack{H \in [\BD_p\text{-gr}] \\ H \gg G,\ H \not\cong G}} \be_H,
\]
and the quotient $\be_G / \textbf{j}_G$ is isomorphic to the simple functor $S^\Delta_{G,\KK}$.

\smallskip
{\rm (ii)} Let $F \subset F'$ be subfunctors of $\KK B$ such that $F'/F$ is simple. Then there exists a unique $G \in [\BD_p\text{-gr}]$ such that $\be_G \subseteq F'$ and $\be_G \nsubseteq F$. In particular, one has
\[
\be_G + F = F', \quad \be_G \cap F = \textbf{j}_G, \quad \text{and} \quad F'/F \cong S^\Delta_{G,\KK}.
\]
\end{proposition}

\begin{corollary}
\label{cor compfactorBurnside}
The composition factors of $\KK B$ as a $p$-bifree biset functor are precisely the simple functors $S^\Delta_{G,\KK}$, where $G$ runs over the set $[\BD_p\text{-gr}]$.
\end{corollary}
To compute the evaluation of the simple functor $S^\Delta_{G,\KK}$ at a finite group $H$, we first describe the structure of $\be_G(H)$ in terms of the subgroup lattice of $H$.
\begin{proposition}
\label{prop basisofbeGH}
Let $G$ and $H$ be finite groups. Then the evaluation $\be_G(H)$ is the subspace of $\KK B(H)$ spanned by the idempotents $e^H_K$, where $K$ runs over a set of representatives of conjugacy classes of subgroups $K \leq H$ satisfying $K \gg \beta(G)$.
\end{proposition}
This description allows one to compute the dimension of $S^\Delta_{G,\KK}(H)$ explicitly in terms of the minimal groups $\beta(K)$ associated to the subgroups $K \leq H$.
\begin{theorem}
\label{thm dimensionofsimples}
Let $G$ be a $\BD_p$-group and $H$ a finite group. Then the $\KK$-dimension of $S^\Delta_{G,\KK}(H)$ is equal to the number of conjugacy classes of subgroups $K \leq H$ such that $\beta(K) \cong G$.
\end{theorem}

\section{The functor $\KK\BD$}
In this section, we study the representable functor  $\KK B^\Delta(-,1)$, which assigns to each finite group $G$ the  $\KK$-vector space $\KK B^\Delta(G,1)$. This functor is a $p$-bifree analogue of the classical representable biset functor $\KK B(-,1)$, and its structure can be analyzed using the same tools developed for the Burnside functor. In particular, we show that its composition factors are parametrized by $B$-groups of $p' $-order, and we describe the corresponding simple functors explicitly. We introduce this functor as a Grothendieck group as follows.

Let $R$ be a commutative ring with identity and let $\KK$ be a field of characteristic zero. Let $G$ be a finite group. A left $G$-set is said to be \textit{left $p$-free}, if it is a disjoint union of transitive $G$-sets with $p'$-stabilizers. Let $\lset{G}^\Delta$ denote the category of left $p$-free $G$-sets. One can similarly define the category $\lset{}_G^\Delta$ of right $p$-free $G$-sets. Identifying a $G$-set $X$ with a $(G,1)$-biset induces an isomorphism $\lset{G}^\Delta\cong \lset{G}_1^\Delta$. 

Let $B^\Delta(G)$ denote the Grothendieck group of the category $\lset{G}^\Delta$. Note that
\begin{align*}
    \BD(G)=\bigoplus_{H\in [s_G]_{p'}} \ZZ[G/H]\,,
\end{align*}
where $[s_G]_{p'}$ denotes a set of representatives of the $G$-conjugacy classes of $p'$-subgroups of $G$.  

\begin{Remark}
    Let $H$ and $K$ be subgroups of $G$. Then
    \begin{align*}
        |(G/K)^H|=|\{x\in G/K \,|\, H^x\subseteq K\}|\,.
    \end{align*}
    In particular, if $K$ is a $p'$-group, i.e., if $[G/K]\in \BD(G)$ and if $H$ is not a $p'$-subgroup, then $|(G/K)^H|=0$. This means that the mark morphism
    \begin{align*}
        \phi: B(G)\to \prod_{H\in [s_G]} \ZZ, \quad [X]\mapsto (|X^H|)_{H\in [s_G]}
    \end{align*}
    restricts to a map
    \begin{align*}
        \phi: \BD(G)\to \prod_{H\in [s_G]_{p'}} \ZZ, \quad [X]\mapsto (|X^H|)_{H\in [s_G]_{p'}}
    \end{align*}
    which we denote again by $\phi$. It follows that the primitive idempotents of $\KK\BD(G)$ are precisely the idempotents $e^G_H\in \KK B(G)$ where $H$ is a $p'$-subgroup of $G$. Setting
    \begin{align*}
        e^G_{p'}:=\sum_{H\in [s_G]_{p'}} e^G_H
    \end{align*}
    we have
    \begin{align*}
        \KK \BD(G)= e_{p'}^G \KK B(G)=\KK B(G)e_{p'}^G=e_{p'}^G\KK B(G)e_{p'}^G\,.
    \end{align*}
\end{Remark}

\begin{Remark}
    Note that the commutative ring $\KK \BD(G)$ does not have an identity element unless $G$ is a $p'$-subgroup.
\end{Remark}

We consider the functor $R\BD:R\calC^\Delta\to \lMod{R}$ defined as
\begin{itemize}
\item $G\mapsto R\BD(G)$.
\item $X\in R\BD(H,G) \mapsto \left(RB(X):R\BD(G)\to R\BD(H),\quad U\mapsto X\cdot_G U\right)$.
\end{itemize}

Note that $R\BD$ is isomorphic to the representable functor $R\BD(-,1)$ at $1$. Also, it is a $p$-bifree subfunctor of $RB$. 

By Theorem~\ref{thm subfunctorlattice}, the poset of the subfunctors of $\KK \BD$ is isomorphic to the poset of closed subsets of 
\begin{align*}
    \Theta(\KK \BD)=\{ H\in[B^\Delta_p\text{-gr}]\,|\, \be_H\subseteq \KK \BD\}\,.
\end{align*}
This observation allows us to identify the minimal groups contributing to the functor  $\KK B^\Delta$. The next result shows that they are exactly the $B$-groups of $p'$-order, in the sense of Bouc’s theory. As a consequence, we obtain the following classification of its composition factors.

\begin{lemma}
    Let $H$ be a $\BD$-group. Then $\be_H\subseteq \KK \BD$ if and only if $H$ is a $p'$-group.
\end{lemma}
\begin{proof}
    Suppose $\be_H\subseteq \KK \BD$. Then $e^H_H\in \be(H)\subseteq \KK \BD(H)$. By the remark above, it follows that $H$ is a $p'$-group. Conversely, if $H$ is a $p'$-group, then $e^H_H\in \KK \BD(H)$ and hence $\be_H\subseteq \KK \BD$.  
\end{proof}

\begin{corollary}\label{cor compfactorsrepresentable}
    The composition factors of the $p$-bifree biset functor $\KK\BD$ are the functors $S^\Delta_{G,\KK}$ where $G$ is a $B$-group of $p'$-order.
\end{corollary}

\section{The simple functor $S_{1,\KK}^\Delta$}
We now turn to the simple functor $S^\Delta_{1,\KK}$, corresponding to the trivial group and the trivial module. This functor plays a foundational role in the theory, as it appears as a composition factor in both the Burnside and character functors. In this section, we compute the dimension of $S^\Delta_{1,\KK}(G)$ for an arbitrary finite group $G$, and show that it is governed by the number of conjugacy classes of cyclic $p'$-subgroups of $G$.

Let $R$ be a commutative ring with identity and let $\KK$ be a field of characteristic zero. Let $(G,V)$ be a pair of a finite group $G$ and a simple $R\Out(G)$-module $V$. We first describe the simple functor $S_{G,V}$ in more detail.

Let $E_G:=R\BD(G,G)$ denote the endomorphism algebra of $G$ in the category $R\calC^\Delta$. Then $\calE^\Delta(G)=E_G/I_G\cong R\Out(G)$. We consider $V$ as a simple $E_G$-module and define the functor $L_{G,V}$ by
\begin{align*}
    L_{G,V}(H)=R\BD(H,G)\otimes_{E_G}V\,.
\end{align*}
Then by \cite{Bouc1996}, $L_{G,V}$ has a unique maximal subfunctor $J_{G,V}$ whose evaluation at a finite group $H$ is given by
\begin{align*}
    J_{G,V}(H)=\Bigl\{\sum_i x_i\otimes v_i\in L_{G,V}(H)\,|\, \forall y\in R\BD(G,H): \sum_i (y\cdot_H x_i)(v_i)=0\Bigr\}\,.
\end{align*}
The simple functor $S_{G,V}^\Delta$ is defined as the quotient $L_{G,V}/J_{G,V}$. 

We compute the $\KK$-dimension of $S_{1,\KK}^\Delta(G)$ inspired by the proof of \cite[Proposition~8]{Bouc1996}. Note that the functor $L_{1,\KK}$ is isomorphic to the functor $\KK\BD$. Moreover for any finite group $G$, identifying $L_{1,\KK}(G)$ with $\KK\BD(G)$ one has
\begin{align*}
    J_{1,\KK}(G)=\{X\in\KK\BD(G)\, |\, \forall Y\in\KK\BD(G): |G\setminus (Y\times X)|=0\}\,.
\end{align*}

This implies that the dimension of $S^\Delta_{1,\KK}(G)$ is equal to the rank of the bilinear form
\begin{align*}
    \langle -,-\rangle: \KK\BD(G)\times\KK\BD(G)\to \KK,\quad (X,Y)\mapsto |G\setminus (Y\times X)|\,.
\end{align*}
The set of primitive idempotents $\{e^G_H\}_{H\in [s_G]_{p'}}$ form an orthogonal basis with respect to this bilinear form. Moreover, for $H\in [s_G]_{p'}$, one has
\begin{align*}
    \langle e^G_H,e^G_H\rangle &=|G\setminus e^G_H|=\frac{1}{|N_G(H)|}\sum_{K\le H}|K|\mu(K,H)\\&
    =\frac{1}{|N_G(H)|} \sum_{x\in H}\sum_{\langle x\rangle\le K\le H}\mu(K,H)=\frac{\phi_1(H)}{|N_G(H)|}
\end{align*}
where $\phi_1(H)$ is the number of elements $x\in H$ such that $\langle x\rangle=H$. This is non-zero if and only if $H$ is cyclic. This proves the following.

\begin{theorem}\label{thm dimensionoftrivialsimple}
    Let $G$ be a finite group. The $\KK$-dimension of $S^\Delta_{1,\KK}(G)$ is equal to the number of conjugacy classes of cyclic $p'$-subgroups of $G$.
\end{theorem}

\begin{Remark}
    Let $G$ be a finite group. Then $\beta(G)=1$ if and only if $G$ is a cyclic $p'$-group. Thus, Theorem~\ref{thm dimensionoftrivialsimple} is a special case of Theorem~\ref{thm dimensionofsimples}.
\end{Remark}

\section{The functor of Brauer characters}
In this section, we study the Brauer character functor $\KK R_k$, which assigns to each finite group $G$ the Grothendieck group of finite-dimensional $kG$-modules (with respect to short exact sequences), tensored with a field $\KK$ of characteristic zero. When $k$ is an algebraically closed field of characteristic $p > 0$, the functor $\KK R_k(-)$ admits a natural structure of a $p$-bifree biset functor. Our goal is to determine its composition factors. It turns out that $\KK R_k(-)$ decomposes as a direct sum of simple $p$-bifree biset functors indexed by cyclic $p'$-groups and primitive modules for their automorphism groups. 


To begin with, if $X$ is a $p$-bifree $(H,G)$-biset, then $kX$ is flat as a right $kG$-module and hence tensoring with $kX$ over $kG$ induces a well-defined group homomorphism
\begin{align*}
    R_k(X):=kX\otimes_{kG}-: R_k(G)\to R_k(H),
\end{align*}
and a $\KK$-linear map
\begin{align*}
    \KK R_k(X): \KK R_k(G)\to \KK R_k(H).
\end{align*}
This endows $\KK R_k(-)$ with a structure of $p$-bifree biset functor. 

\begin{Remark}
    Recall that one has an isomorphism $\KK R_k(-)\cong \KK \Proj(-)\cong S_{1,1,\KK}$ of simple diagonal $p$-permutation functors, see \cite[Theorem~5.20]{BoucYilmaz2020} and \cite[Section~6]{BoucYilmaz2025}. 
\end{Remark}

Our aim is to describe the composition factors of $\KK R_k(-)$ as a $p$-bifree biset functor. We first determine the restriction kernels.

\begin{proposition}\label{prop:restkernel}
Let $G$ be a finite group.

\smallskip
{\rm (i)} One has $\underline{\mathbb K R_k}(G) = 0$ unless $G$ is a cyclic $p'$-group.

\smallskip
{\rm (ii)} Let $G$ be a cyclic $p'$-group. Then the $\mathbb K\Aut(G)$-module $\underline{\mathbb K R_k}(G)$ is isomorphic to the direct sum of primitive $\mathbb K\Aut(G)$-modules.
\end{proposition}
\begin{proof}
    (i) Let $\chi\in\mathbb KR_k(G)$ be a non-zero element and let $g\in G$ with $\chi(g)\not= 0$. Then $g$ is a $p'$-element and $\Res^G_{\langle g\rangle}\chi\not= 0$. This shows that $\chi\not\in\underline{\mathbb K R_k}(G)$ and therefore, $\underline{\mathbb K R_k}(G) = 0$ if $G$ is not a cyclic $p'$-group.

    (ii) Let $G$ be a cyclic $p'$-group. We may identify $\mathbb KR_k(H)$ with $\mathbb KR_\mathbb C(H)$ and $\underline{\KK R_k}(H)$ with $\underline{\KK R_\CC} (H)$ by Lemma~\ref{lem: restrictionkernelsection}. The result follows from Corollary 7.3.5 \cite{Bouc2010a}.
\end{proof}

The next step is to show that every element of  $\KK R_k(G)$ lies, up to action by an element of the ideal $I_G$, in the restriction kernel. This will allow us to apply the general theory of restriction kernels and deduce composition factor multiplicities.
\begin{proposition}\label{prop:condition}
    Let $G$ be a finite group. For any $\chi\in \mathbb K R_k(G)$, there exists $\chi'\in \underline{\mathbb K R_k}(G)$ and $\alpha\in I_G$ satisfying the equality 
    \begin{align*}
    \chi = \chi' + \alpha\cdot \chi\,.
    \end{align*}
\end{proposition}
\begin{proof}
    By Proposition 2.5 (iv) in \cite{BHabil}, for any $\chi\in \mathbb K R_k(G)$, we have
    \[
    \chi = \frac{1}{|G|}\sum_{\substack{L\le K\le G\\ K \, cyclic,\, |K|_p=1}}|L|\mu(L,K)\ind_L^G\res^G_L\chi\,.
    \]
    It follows that if $G$ is not a cyclic $p'$-group, the claim holds by putting $\chi'=0$ and setting 
    \[
    \alpha = \frac{1}{|G|}\sum_{\substack{L\le K\le G\\ K \, cyclic,\, |K|_p=1}}|L|\mu(L,K) \Big[ \frac{G\times G}{\Delta(L)} \Big]\in I_G.
    \] 
    Now suppose $G$ is a cyclic $p'$-group. The restriction of $\mathbb KR_k$ to the full subcategory of cyclic $p'$-groups is semisimple and hence the result follows from Proposition~\ref{prop: semisimpledecomposition}.
    \end{proof}

With this result, we can now describe precisely when a simple $p$-bifree biset functor $S^\Delta_{H,V}$ appears as a composition factor of $KR_k$, and compute its multiplicity.

\begin{theorem}\label{thm:Brauerchars}
    The simple $p$-bifree biset functor $S_{H, V}^\Delta$ is a composition factor of the functor $\mathbb K R_k$ if and only if $H$ is a cyclic $p'$-group and $V$ is a primitive $\mathbb K\Aut(H)$-module. In this case, the multiplicity of $S_{H, V}^\Delta$ is equal to one.
\end{theorem}
\begin{proof}
     It follows from Theorem~\ref{thm thecondition} and Proposition \ref{prop:condition} that the multiplicity of $S_{H, V}^\Delta$ in $\mathbb KR_k$ as a composition factor is equal to that of the $\mathbb K\Aut(G)$-module $V$ in $\underline{\mathbb K R_k}(H)$. 
    
    By Proposition \ref{prop:restkernel}, the restriction kernel is zero unless $H$ is a cyclic $p'$-group. Hence the multiplicity of $S_{H, V}^\Delta$ in $\KK R_k$ is equal to $0$ if $H$ is not a cyclic $p'$-group. Furthermore, if $H$ is a cyclic $p'$-group, by Proposition~\ref{prop:restkernel}{\rm (ii)}, the multiplicity of $S_{H, V}^\Delta$ in $\KK R_k$ is non-zero if and only if $V$ is primitive, in which case it is equal to $1$. 
\end{proof}
\begin{proposition}\label{prop simpleprimitiveBrauer}
    Let $H$ be a cyclic $p'$-group and $V$ a primitive $\KK\Aut(H)$-module. The subfunctor $F$ of $\KK R_k$ generated by $V$ is isomorphic to the simple functor $S^\Delta_{H,V}$.
\end{proposition}
\begin{proof}
    First note that $F(H)\cong V$ and $H$ is a minimal group for $F$ since $V$ is in the restriction kernel. Suppose for a contradiction that $F$ is not simple. Then there exists a composition factor $T$ of $F$ not isomorphic to $S^\Delta_{H,V}$. By Theorem~\ref{thm:Brauerchars}, $T\cong S^\Delta_{H',V'}$ for some cyclic $p'$-group $H'$ and a primitive $\KK\Aut(H')$-module $V'$. Note that since $F$ is generated by $V$, one has $|H|<|H'|$. Note also that the functor $F$ is a quotient of the functor $\Delta_{H,V}$ introduced in Section~5 of \cite{Webb2010}. 
    
    Consider the full subcategory $\mathcal{C}_{H'}$ on the set of subquotients of $H'$. By Corollary~5.4 in \cite{Webb2010}, the restriction of the functor $\Delta_{H,V}$ to $\mathcal{C}_{H'}$ is isomorphic to the functor $\lexp{\mathcal{C}_{H'}}{\Delta_{H,V}}$. But the category of $p$-bifree biset functors over $\mathcal{C}_{H'}$ is the same as the category of biset functors over $\mathcal{C}_{H'}$ and hence it is semisimple. In particular, the restriction of the functor $\Delta_{H,V}$ to $\mathcal{C}_{H'}$ is simple as $\Delta_{H,V}$ has a simple head by Corollary~5.5 in \cite{Webb2010}. This implies that the restriction of $F$ to $\mathcal{C}_{H'}$ is a simple functor with minimal group $H$. On the other hand, the restriction of $T$ appears as a composition factor in the restriction of $F$. Contradiction.
\end{proof} 

\begin{corollary}\label{cor Brauercharacter}
    One has
    \begin{align*}
        \KK R_k\cong \bigoplus_{(m,\xi)}S^\Delta_{\ZZ/m\ZZ, \CC_\xi}
    \end{align*}
    where $(m,\xi)$ runs through a set of pairs consisting of a positive $p'$-number $m$ and a primitive character $\xi$ of $(\ZZ/m\ZZ)^\times$.
\end{corollary}
\begin{proof}
    This follows immediately from Theorem~\ref{thm:Brauerchars} and Proposition~\ref{prop simpleprimitiveBrauer}
\end{proof}

\section{The functor of complex character ring}
In this section, we examine the ordinary character ring functor $\CC R_\CC$ as a $p$-bifree biset functor. Since the ordinary character ring is additive under biset composition and functorial with respect to maps in the $p$-bifree category, it inherits a natural $p$-bifree structure. We classify its composition factors explicitly and describe its relationship to the Brauer character functor $\KK R_k$. The main result is a short exact sequence of $p$-bifree biset functors, which reveals how the modular and ordinary character theories fit together within this categorical framework.

Let $\CC$ be an algebraically closed field of characteristic zero. For a finite group $G$, we denote by $R_\CC(G)$ the Grothendieck ring of the category of finite dimensional $\CC G$-modules. We identify it with the ring of class functions on $G$ with values in $\CC$. We set $\CC R_\CC(G)=\CC\otimes_\ZZ R_\CC(G)$.

If $H$ is another finite group and $X$ is a $p$-bifree $(H,G)$-biset, then tensoring with $\CC X$ over $\CC G$ induces a well-defined $\CC$-linear map
\begin{align*}
    \CC R_\CC(X): \CC R_\CC(G)\to \CC R_\CC(H).
\end{align*}
This endows $\CC R_\CC(-)$ with a structure of $p$-bifree biset functor. 

\begin{Remark}\label{rem directsumcharacterfunctor}
Given a pair $(G,V)$ where $G$ is a finite group and $V$ is a simple $\CC\Out(G)$-module, let $S_{G,V}$ denote the associated simple biset functor. By \cite[Corollary~7.3.5]{Bouc2010a}, one has an isomorphism
\begin{align*}
    \CC R_\CC\cong \bigoplus_{(m,\xi)}S_{\ZZ/m\ZZ,\CC_\xi}
\end{align*}
of biset functors, where $(m,\xi)$ runs through the set of pairs consisting of a positive integer $m$ and a primitive character $\xi:(\ZZ/m\ZZ)^\times\to\CC^\times$. This restricts to an isomorphism of $p$-bifree biset functors. Note that $S_{\ZZ/m\ZZ,\CC_\xi}$ is not necessarily simple as a $p$-bifree biset  functor. For any finite group $G$, we identify $S_{\ZZ/m\ZZ,\CC_\xi}(G)$ as a subspace of $\CC R_\CC(G)$ via this isomorphism. We also denote by $\tilde{\xi}$ the (virtual) character of $\ZZ/m\ZZ$ obtained by extending $\xi$ by $0$, i.e., for $x\in \ZZ/m\ZZ$,
\begin{align*}
    \tilde{\xi}(x)=\begin{cases}
        \xi(x), \quad &\text{if } x\in (\ZZ/m\ZZ)^\times\\
        0,\quad &\text{otherwise}\,.
    \end{cases}
\end{align*}
\end{Remark}

\begin{lemma}\label{lem conditionforprimitive}
    Let $G$ be a finite group. Let also $m$ be a positive integer and $\xi$ a primitive character of $(\ZZ/m\ZZ)^\times$. Then for any $\chi\in S_{\ZZ/m\ZZ,\CC_\xi}(G)$, there exists $\chi'\in \underline{S_{\ZZ/m\ZZ,\CC_\xi}}(G)$ and $\alpha\in I_G$ such that
    \begin{align*}
        \chi=\chi'+\alpha\cdot\chi\,.
    \end{align*}
\end{lemma}
\begin{proof}
    By Proposition~1.5(iv) of Chapter 3 of \cite{BHabil}, for any $\chi\in \CC R_\CC(G)$, we have
    \begin{align*}
    \chi = \frac{1}{|G|}\sum_{\substack{L\le K\le G\\ K \, cyclic}}|L|\mu(L,K)\ind_L^G\res^G_L\chi\,.
    \end{align*}
    It follows that if $G$ is not a cyclic group, the claim holds by putting $\chi'=0$ and setting 
    \[
    \alpha = \frac{1}{|G|}\sum_{\substack{L\le K\le G\\ K \, cyclic}}|L|\mu(L,K) \Big[ \frac{G\times G}{\Delta(L)} \Big]\in I_G.
    \] 
    Now suppose $G=\ZZ/n\ZZ$ is a cyclic group. If $n$ is not a multiple of $m$, then by \cite[Corollary~7.4.3]{Bouc2010a} $S_{\ZZ/m\ZZ,\CC_\xi}(\ZZ/n\ZZ)=\{0\}$ and hence there is nothing to prove. Therefore, assume that $n$ is a multiple of $m$ and let $\chi\in S_{\ZZ/m\ZZ,\CC_\xi}(\ZZ/n\ZZ)$. Then we can write
    \[
    \chi = \tilde e_G^G\chi + (1-\tilde e_G^G)\chi.
    \]
    Since $1-\tilde e_G^G\in I_G$, it is sufficient to consider $\tilde e_G^G\chi$. But since $S_{\mathbb Z/m\mathbb Z,\CC_\xi}$ is generated by $\tilde \xi$, the space $\tilde e_G^GS_{\mathbb Z/m\mathbb Z, \CC_\xi}(G)$ is one dimensional generated by the inflation $\tilde e_G^G\Inf_{\mathbb Z/m\mathbb Z}^G\tilde \xi$. Hence we may put $\tilde e_G^G\chi = \tilde e_G^G\Inf_{\mathbb Z/m\mathbb Z}^G\tilde \xi$. 

    Let $N\le G$ be such that $G/N \cong \mathbb Z/m\mathbb Z$. Write $N = N_p\times N_{p'}$. Assume first that $N_{p'}\not= 1$. Then, by \cite[Proposition~2.5.12 and Theorem~5.2.4]{Bouc2010a}, one has
    \begin{eqnarray*}
    (\tilde e_G^G\Inf_{G/N_{p'}}^G\Def_{G/N_{p'}}^G) \tilde e_G^G\Inf_{G/N}^G\tilde \xi &=& \tilde e_G^G\Inf_{G/N_{p'}}^G(\Def_{G/N_{p'}}^G \tilde e_G^G\Inf_{G/N_{p'}}^G)\Inf_{G/N}^{G/N_{p'}}\tilde \xi \\
    &=& m_{G, N_{p'}}\tilde e_G^G\Inf_{G/N_{p'}}^G\tilde e_{G/N_{p'}}^{G/N_{p'}}\Inf_{G/N}^{G/N_{p'}}\tilde \xi\\
    &=& m_{G, N_{p'}}\tilde e_G^G\widetilde{(\Inf_{G/N_{p'}}^G e_{G/N_{p'}}^{G/N_{p'}})}\Inf_{G/N_{p'}}^G\Inf_{G/N}^{G/N_{p'}}\tilde \xi\\
    &=& m_{G, N_{p'}}\tilde e_G^G\Inf_{G/N}^G\tilde \xi\,.
    \end{eqnarray*}
    
Since $G$ is cyclic, $m_{G, N_{p'}}\not = 0$ by \cite[Proposition~5.6.1]{Bouc2010a}. Thus, the condition is satisfied by putting $\chi' = 0$ and $\alpha =\frac{1}{m_{G,N}} \tilde e_G^G\Inf_{G/N_{p'}}^G\Def_{G/N_{p'}}^G$. This also shows that $\underline{S_{\ZZ/m\ZZ, \CC_\xi}}(G) = 0$. 
    
Now suppose that $N_{p'} = 1$, so $N$ is a cyclic $p$-group. Note that the restriction of $\tilde e_G^G\Inf^G_{G/N}\tilde\xi$ to a proper subgroup is equal to zero. Hence if $G$ is a $p$-group, then $\tilde e_G^G\Inf^G_{G/N}\tilde\xi\in \underline{S_{\ZZ/m\ZZ, \CC_\xi}}(G)$. Again, the condition is satisfied in this case too.
    
Suppose $G$ is not a $p$-group, then
    \begin{eqnarray*}
        \tilde e_G^G\Inf_{G/N}^G\tilde\xi &=& \frac{1}{|G|}\sum_{K\le G}|K| \mu(K, G)\Ind_K^G\Res^G_K\Inf_{G/N}^G\tilde\xi\\
        &=& \frac{1}{|G|}\sum_{K\le G}|K| \mu(K, G)\Ind_K^G\Inf^K_{K/(K\cap N)}\Isom\Res^{G/N}_{KN/N}\tilde\xi\\
        &=& \frac{1}{|G|}\sum_{K\le G: KN = G}|K| \mu(K, G)\Ind_K^G\Inf^K_{K/(K\cap N)}\Isom \,\tilde \xi\\
    \end{eqnarray*}
    
    Note that since $N$ is a $p$-subgroup of $G$, the condition $KN= G$ holds if and only of $G_{p'}\subseteq K$ and $K_pN = G_p$. There are two cases:
    \begin{enumerate}[(i)]
        \item If $N<G_p$, then $NK_p = G_p$ holds if and only if $K_p = G_p$ and in this case we get $K = G$.
        \item If $N = G_p$, then $K_p$ can be any subgroup of $G_p$, so $K$ is any subgroup with $G_{p'}\subseteq K$.
    \end{enumerate}

    In Case (i), we get
    \begin{eqnarray*}
        \tilde e_G^G\Inf_{G/N}^G\tilde\xi 
        &=& \frac{1}{|G|}\sum_{K\le G: KN = G}|K| \mu(K, G)\Ind_K^G\Inf^K_{K/(K\cap N)}\Isom \,\tilde \xi\\
        &=& \Inf_{G/N}^G\tilde\xi
    \end{eqnarray*}
    On the other hand, in Case (ii), we obtain
    \begin{eqnarray*}
        \tilde e_G^G\Inf_{G/N}^G\tilde \xi 
        &=& \frac{1}{|G|}\sum_{G_{p'}\le K\le G}|K| \mu(K, G)\Ind_K^G\Inf^K_{K/(K\cap N)}\Isom \,\tilde\xi\,.\\
    \end{eqnarray*}
    Here we have $\mu(K, G) = \mu(|G:K|)$, the number theoretic Mobius function. We have $\mu(|G:K|) = 1$ if $G=K$, it is $-1$ if $|G:K| = p$ and $0$ otherwise. Hence the above equality becomes
      \begin{eqnarray*}
        \tilde e_G^G\Inf_{G/N}^G\tilde \xi 
        &=& \Inf_{G/N}^G\Isom \,\tilde\xi - \frac{1}{p}\Ind\Inf_{K/(K\cap N)}^G\Isom\, \tilde\xi\,. \\
    \end{eqnarray*}
    
Now for any non-trivial $p'$-subgroup $M$ of $G$, we have
    \[
    \Def^G_{G/M}\Inf_{G/N}^G\tilde \xi = \Inf_{G/NM}^{G/M}\Def^{G/N}_{G/NM}\tilde \xi = 0
    \]
    since $|G:NM|<|G:N| = m$. Hence $\tilde e^G_G\Inf_{G/N}^G\tilde \xi\in\underline{S_{\ZZ/m\ZZ, \CC_\xi}}(G)$ in Case (i).

    For Case (ii), we need to evaluate 
    \[
    \Def^G_{G/M}\Big(\Inf_{G/N}^G\Isom \,\xi - \frac{1}{p}\Ind\Inf_{K/(K\cap N)}^G\Isom \xi\Big)\,.
    \]
The first term is zero, by the above calculations. We also evaluate the second term:
\[
    \Def^G_{G/M}\Ind\Inf_{K/(K\cap N)}^G\Isom\, \tilde\xi = \Ind\Inf_{K/(M(K\cap N))}^{G/M}\Def^{K/(K\cap N)}_{K/(M(K\cap N))}\Isom\,\tilde \xi = 0
    \]
Hence $\tilde e^G_G\Inf_{G/N}^G\tilde \xi\in\underline{S_{\ZZ/m\ZZ, \CC_\xi}}(G)$ in Case (ii) too. The result follows. 
\end{proof}

\begin{notation}
    Let $m\in \NN$ and let $\xi$ be a primitive character of $(\ZZ/m\ZZ)^\times$. Let $n\in\NN_0$ and let $G=\ZZ/mp^n\ZZ$ be the cyclic group of order $m\cdot p^n$. We set $\xi_{m,n,p}=\Inf^{G^\times}_{(\ZZ/m\ZZ)^\times} \xi$ and $\tilde\xi_{m,n,p}:=\tilde e^G_G\Inf^G_{\ZZ/m\ZZ}\tilde\xi$.
\end{notation}
Note that for $g\in G$, one has
\begin{align*}
    \tilde\xi_{m,n,p}(x)=\begin{cases}
        \Inf^G_{\ZZ/m\ZZ}\tilde\xi(x), \quad &\text{if } \langle x\rangle=G\,,\\
        0, \quad &\text{otherwise}. 
    \end{cases}
\end{align*}
Thus, $\tilde\xi_{m,n,p}$ is the Dirichlet character of $G$ induced by $\xi_{m,n,p}$. Since $\xi$ is primitive, it follows that the conductor of $\xi_{m,n,p}$ is equal to $m$. 

\begin{corollary}\label{cor restkerofprimbiset}
    Let $G$ be a finite group, $m$ a positive integer, and $\xi$ a primitive character of $(\ZZ/m\ZZ)^\times$.

    \smallskip
    {\rm (a)} $\underline{S_{\ZZ/m\ZZ,\CC_\xi}}(G)=\{0\}$ unless $G$ is a cyclic group of order $mp^n$ for some integer $n\ge 0$.

    {\rm (b)} If $G$ is a cyclic group of order $mp^n$ for some $n\ge 0$, then $\underline{S_{\ZZ/m\ZZ,\CC_\xi}}(G)$ is one-dimensional generated by $\tilde\xi_{m,n,p}$.
\end{corollary}
\begin{proof}
    Both parts follow from the proof of Lemma~\ref{lem conditionforprimitive}. 
\end{proof}

Now we show that the functor $S_{\ZZ/m\ZZ,\CC_\xi}$ is semisimple as a $p$-bifree biset functor.

\begin{theorem}\label{thm decompofsimplebisetfun}
    Let $m$ be a positive integer and let $\xi$ be a primitive character of $(\ZZ/m\ZZ)^\times$. Then
    \begin{align*}
        S_{\ZZ/m\ZZ,\CC_\xi}\cong \bigoplus_{n\in \NN_0} S^\Delta_{\ZZ/mp^n\ZZ,\CC_{\xi_{m,n,p}}}
    \end{align*}
    as $p$-bifree biset functors.
\end{theorem}
\begin{proof}
    As before we regard $S:=S_{\ZZ/m\ZZ,\xi}$ as a subfunctor of $\CC R_\CC$. By Lemma~\ref{lem conditionforprimitive}, Corollary~\ref{cor restkerofprimbiset} and Theorem~\ref{thm thecondition}, the composition factors of $S$ are exactly the functors $S^\Delta_{\ZZ/mp^n\ZZ,\CC_{\xi_{m,n,p}}}$ each with multiplicity one. We will show that each of these functors appears as a subfunctor. Let $F$ be a $p$-bifree biset subfunctor of $S$ and let $G$ be a minimal group for $F$. Since $F(K)=0$ for any group $K$ with $|K|<|G|$, we have $0\neq F(G)\subseteq \underline{S}(G)$. By Corollary~\ref{cor restkerofprimbiset}, it follows that $G$ is a cyclic group of order $m\cdot p^n$ for some natural number $n$, and that $F(G)$ is one-dimensional generated by $\tilde\xi_{m,n,p}=\tilde{e}^G_G\Inf^G_{\ZZ/m\ZZ}\xi$. Conversely, if $G$ is a cyclic group of order $m\cdot p^n$, then $\underline{S(G)}\neq 0$ and the subfunctor generated by $\underline{S}(G)$ has a minimal group $G$.

    For any natural number $n\ge 0$, let $G_n=\ZZ/mp^n\ZZ$ be a cyclic group of order $m\cdot p^n$ and let $F_n$ be the subfunctor of $S$ generated by $\tilde\xi_{m,n,p}$. Then by discussions above $F_n$ has a minimal group $G_n$. We claim that $F_n$ is a simple $p$-bifree biset functor. It is sufficient to show that $F_{n+k}\not\subseteq F_n$ for any $k\ge 1$. Indeed, if $F_n$ has a non-zero subfunctor $M$, then the minimal group of $M$ is $G_{n+k}$ for some $k$ and by definition $F_{n+k}\subseteq M$. 

    Let $k\ge 1$. Since $|G_{n+k}|/|G_n|$ is a $p$-power, the space $\CC\BD(G_{n+k},G_n)\circ F_n(G_n)$ is one-dimensional, generated by $\Ind^{G_{n+k}}_{G_n} \tilde\xi_{m,n,p}$. In particular, $\tilde{e}^{G_{n+k}}_{G_{n+k}}F_n(G_{n+k})=0$. But, $\tilde{e}^{G_{n+k}}_{G_{n+k}}F_{n+k}(G_{n+k})=F_{n+k}(G_{n+k})\neq 0$. Hence, $F_{n+k}\not\subseteq F_n$, as required. This shows that $F_n$ is simple. Moreover, for any $x\in (\ZZ/mp^n)^\times$, we have $x\cdot\tilde\xi_{m,n,p}=\xi_{m,n,p}(x)\tilde\xi_{m,n,p}$ which implies $F_n(G_n)\cong \CC_{\xi_{m,n,p}}$. This shows that $F_n\cong S^\Delta_{G_n, \CC_{\xi_{m,n,p}}}$ and theorem is proved.
\end{proof}

These calculations show that $\CC R_\CC$ decomposes into a direct sum of simple $p$-bifree functors indexed by pairs  $(m, \xi)$, where $m$ is a positive integer and $\xi$ is an irreducible character of $(\mathbb{Z}/m\mathbb{Z})^\times$ with primitive $p'$-part.

\begin{corollary}\label{cor complexcharacter}
    One has
    \begin{align*}
        \CC R_\CC\cong \bigoplus_{(m,\xi)} S^\Delta_{\ZZ/m\ZZ, \CC_\xi}
    \end{align*}
    where $(m,\xi)$ runs through the set of pairs consisting of a positive integer $m$ and a character $\xi$ of $(\ZZ/m\ZZ)^\times$ such that $\xi_{p'}$ is primitive. 
\end{corollary}
\begin{proof}
    By Remark~\ref{rem directsumcharacterfunctor} and Theorem~\ref{thm decompofsimplebisetfun}, we have
   \begin{align*}
       \CC R_\CC\cong \bigoplus_{(k,\chi)}\bigoplus_{n\in \NN_0} S^\Delta_{\ZZ/kp^n\ZZ, \CC_{\chi_{m,n,p}}}
   \end{align*} 
   where $(k,\chi)$ runs through the set of pairs consisting of a positive integer $k$ and a primitive character $\chi$ of $(\ZZ/k\ZZ)^\times$. Therefore,
      \begin{align*}
       \CC R_\CC\cong \bigoplus_{m\in\NN}\bigoplus_{\xi} S^\Delta_{\ZZ/m\ZZ, \CC_\xi}
   \end{align*} 
   where $\xi$ runs through the characters of $(\ZZ/m\ZZ)^\times$ such that the quotient of $m$ by the conductor of $\xi$ is a $p$-power. The result follows.
\end{proof}

\begin{Remark}
    Given finite groups $G$ and $H$, we can choose a $p$-modular system $(\KK,\calO,k)$, where $\calO$ is a complete discrete valuation ring, $\KK$ is the quotient field of $\calO$ of characteristic zero containing a primitive $\exp (G\times H)$-th root of unity and $k$ is the residue field of $\calO$ which is algebraically closed. We also choose an embedding $\KK \to \CC$. Then, by Proposition~4.17.4 in \cite{Linckelmann2018a}, for any $X\in \CC \BD(H,G)$, the diagram 
    
    \begin{center}
\begin{tikzcd} 
  \CC R_\CC(G) \arrow{r}{d_G}\arrow{d}[swap]{\CC R_\CC(X)} 
  & \CC R_k(G) \arrow{d}{\CC R_k(X)} 
  \\
  \CC R_\CC(H) \arrow{r}[swap]{d_H} 
  & \CC R_k(H) 
\end{tikzcd}
\end{center}
where $d_G$ and $d_H$ denote the decomposition maps, commute. This way we obtain a natural transformation $d: \CC R_\CC\to \CC R_k$. By Corollaries~\ref{cor Brauercharacter} and \ref{cor complexcharacter}, we obtain a short exact sequence
\begin{align*}
    0\longrightarrow \bigoplus_{(m,\xi)} S^\Delta_{\ZZ/m\ZZ,\CC_\xi}\longrightarrow \CC R_\CC \longrightarrow \CC R_k\longrightarrow 0
\end{align*}
of $p$-bifree biset functors, where $(m,\xi)$ runs through the set of pairs consisting of a positive integer $m$ divisible by $p$ and an irreducible character $\xi$ of $(\ZZ/m\ZZ)^\times$ with primitive $p'$-part.
\end{Remark}

\section{Further examples}
In this section, we analyze how certain classical simple biset functors behave under restriction to the $p$-bifree biset category. Our focus is on the simple functors $S_{C_p \times C_p, \CC}$ and $S_{C_q \times C_q, \CC}$, for a prime number $q$ with $q\neq p$, whose evaluations are closely related to the structure of the Dade group (see \cite[Chapters~11--12]{Bouc2010a}). It turns out that both functors decompose as direct sums of simple $p$-bifree biset functors indexed by specific families of $p$-groups. 


We begin with the case of $S_{C_p \times C_p, \CC}$, the simple biset functor associated to the elementary abelian $p$-group of rank two. We determine the groups $G$ for which its restriction kernel evaluates non-trivially in the $p$-bifree category and compute it when non-zero.

\begin{proposition}\label{prop restkerelemabelian}
    {\rm (a)} Let $G$ be a finite group. One has $\underline{S_{C_p\times C_p,\CC}}(G)=\{0\}$ unless $G$ is a non-cyclic $p$-group.

    \smallskip
    {\rm (b)} Let $G$ be a non-cyclic $p$-group. Then $\underline{S_{C_p\times C_p,\CC}}(G)$ is one-dimensional generated by the image of $e^G_G$ in $S_{C_p\times C_p,\CC}(G)$.
\end{proposition}
\begin{proof}
    First note that for any $p$-bifree biset functor $S$ and finite group $G$, one has $\underline{S}(G)\subseteq \tilde e^G_G S(G)$. Indeed, one has $S(G)=\tilde e^G_G S(G)\oplus (1-\tilde e^G_G) S(G)$ and $(1-\tilde e^G_G)S(G)\cap \underline{S}(G)=\{0\}$.

    Now note that by \cite[Corollary~3.6]{Bouc2023}, $\tilde e^G_G S_{C_p\times C_p,\CC}(G)$ is non-zero if and only if $G$ is $p$-elementary. Let $\underline{\be}_{C_p\times C_p}$ and $\underline{\textbf{j}}_{C_p\times C_p}$ be biset subfunctors of the Burnside ring functor introduced in \cite[Section~5]{Bouc2010a}, similar to the functors $\be_{C_p\times C_p}$ and $\textbf{j}_{C_p\times C_p}$ in Section~4. By the proof of \cite[Theorem~5.5.4]{Bouc2010a}, the functor $S_{C_p\times C_p,\CC}$ is isomorphic as biset functors to the quotient $\underline{\be}_{C_p\times C_p}/\underline{\textbf{j}}_{C_p\times C_p}$. Furthermore, $\tilde e^G_GS_{C_p\times C_p,\CC}(G)$ is generated by the image $\overline{e^G_G}$ in $S_{C_p\times C_p,\CC}(G)$ of $e^G_G$ if $G$ is not cyclic and is equal to zero if $G$ is cyclic. 

    Suppose that $G=P\times C$ is a non-cyclic $p$-elementary group where $P=O_p(G)$. Then $e^G_G=e^P_P\times e^C_C$. Also if $N\unlhd G$ is of $p'$-order, then $G/N\cong P\times C/N$ and
    \begin{align*}
        \Def^G_{G/N}e^G_G=e^P_P\times \Def^C_{C/N}e^C_C\,.
    \end{align*}
    Since $C$ is a cyclic group, $m_{C,C}\neq 0$ and hence
        \begin{align*}
        \Def^G_{G/C}e^G_G=m_{C,C} e^P_P\times e^{C/C}_{C/C}\neq 0\,.
    \end{align*}
    This shows that $\underline{S_{C_p\times C_p,\CC}}(G)=\{0\}$ unless $C=1$. On the other hand if $C=1$, then $\underline{S_{C_p\times C_p,\CC}}(G)=\tilde e^G_GS_{C_p\times C_p,\CC}(G)$ is one-dimensional generated by $\overline{e^G_G}$. This proves both parts.
\end{proof}

\begin{proposition}\label{prop conditionelemabelian}
    Let $G$ be a finite group. Then for any $x\in S_{C_p\times C_p,\CC}(G)$, there exists $x'\in S_{C_p\times C_p,\CC}(G)$ and $\alpha\in I_G$ such that $x=x'+\alpha\cdot x$.
\end{proposition}
\begin{proof}
    Let $x\in S_{C_p\times C_p,\CC}(G)$. By the proof of Proposition~\ref{prop restkerelemabelian}, we have $\tilde e^G_G S_{C_p\times C_p,\CC}(G)=\{0\}$ if $G$ is not a non-cyclic $p$-group. In this case, we have $x'=0$ and $\alpha=1-\tilde e^G_G$. If $G$ is a non-cyclic $p$-group, then again by the proof of Proposition~\ref{prop restkerelemabelian}, we have $\underline{S_{C_p\times C_p,\CC}}(G)=\tilde e^G_GS_{C_p\times C_p,\CC}(G)$. Thus, in this case, we choose $x'=\tilde e^G_G\cdot x'$ and $\alpha=1-\tilde e^G_G$.
\end{proof}

These two results allow us to determine the full decomposition of $S_{C_p \times C_p, \CC}$ as a $p$-bifree biset functor. It splits as a direct sum of simple functors indexed by non-cyclic $p$-groups, each occurring with multiplicity one.
\begin{theorem}\label{thm elemabelian}
    One has an isomorphism
    \begin{align*}
        S_{C_p\times C_p,\CC}\cong \bigoplus_{P: \text{non-cyclic } p\text{-group}} S^\Delta_{P,\CC}
    \end{align*}
    of $p$-bifree biset functors.
\end{theorem}
\begin{proof}
    By Propositions~\ref{prop restkerelemabelian} and \ref{prop conditionelemabelian} and Theorem~\ref{thm thecondition}, the composition factors of $S_{C_p\times C_p,\CC}$ are exactly the functors $S_{P,\CC}^\Delta$, where $P$ is a non-cyclic $p$-group, each with multiplicity one. We will show that each of these functors is a subfunctor.

    Let $F$ be a $p$-bifree biset subfunctor of $S:=S_{C_p\times C_p,\CC}$ and let $G$ be a minimal group of $F$. Then we have $\{0\}\neq F(G)\subseteq \underline{S}(G)$ which by Proposition~\ref{prop restkerelemabelian} implies that $G$ is a non-cyclic $p$-group and that $F(G)$ is one-dimensional generated by $\overline{e^G_G}$. Conversely, if $G$ is a non-cyclic $p$-group, then the subfunctor generated by $\underline{S}(G)$ has a minimal group $G$. 

    For a non-cyclic $p$-group $P$, let $F_P$ be the subfunctor of $S$ generated by $\overline{e^P_P}$. Then $F_P$ has a minimal group $P$. Moreover, if $P'$ is a $p$-group such that $F_P(P')\neq 0$, then $P$ is isomorphic to a subgroup of $P'$. Indeed, if $X$ is a transitive $p$-bifree $(P',P)$-biset such that $X\circ e^P_P\neq 0$, then $p_2(X)=P$ and $k_2(X)=1=k_1(X)$. 

    We claim that $F_P$ is a simple $p$-bifree biset functor isomorphic to $S^\Delta_{P,\CC}$. To show that it suffices to prove that if $P'$ is a non-cyclic $p$-group not isomorphic to $P$, then $F_{P'}\not\subseteq F_P$. By above $F_P(P')=0$ if $P$ is not isomorphic to a subgroup of $P'$. Also, if $P$ is isomorphic to a subgroup of $P'$, then $F_P(P')$ is one-dimensional generated by $\Ind^{P'}_P e^P_P$, or by $e^{P'}_P$. In both cases, it follows that if $P\not\cong P'$, then $e^{P'}_{P'} F_P(P')=0$. But $e^{P'}_{P'} F_{P'}(P')=F_{P'}(P')\neq 0$. Thus, $F_{P'}\not\subseteq F_P$ and hence $F_P$ is simple. Since $F_P(P)$ is one-dimensional and generated by $\overline{e^P_P}$ which is invariant under any automorphism of $P$, we have $F_P(P)\cong \CC$. Therefore, $F_P\cong S^\Delta_{P,\CC}$. This proves the theorem. 
\end{proof}

We now turn to the case of $S_{C_q \times C_q,\CC}$ for a prime $q \ne p$. Here, the restriction kernel of the functor is more restricted. We identify the finite groups $G$ for which it is non zero and determine the associated simple components.
\begin{proposition}\label{prop restkerqelementary}
    {\rm (a)} Let $G$ be a finite group. One has $\underline{S_{C_q\times C_q,\CC}}(G)=\{0\}$ unless $G$ is isomorphic to $C_q\times C_q\times P$ for some cyclic $p$-group.

    \smallskip
    {\rm (b)} Let $P$ be a cyclic $p$-group and let $G=C_q\times C_q\times P$. Then $\underline{S_{C_q\times C_q,\CC}}(G)$ is one-dimensional generated by image of $e^G_G$ in $S_{C_q\times C_q,\CC}(G)$.
\end{proposition}
\begin{proof}
        The proof is similar to the proof of Proposition~\ref{prop restkerelemabelian}. By \cite[Corollary~3.6]{Bouc2023}, $\tilde e^G_G S_{C_q\times C_q,\CC}(G)$ is non-zero if and only if $G$ is $q$-elementary. Furthermore, in this case, by Theorem~5.5.4 in \cite{Bouc2010a}, $\tilde e^G_GS_{C_q\times C_q,\CC}(G)$ is generated by the image $\overline{e^G_G}$ in $S_{C_q\times C_q,\CC}(G)$ of $e^G_G$ if $G$ is not cyclic and is equal to zero if $G$ is cyclic. 

        Suppose $G=Q\times C$ is a non-cyclic $q$-elementary group where $Q=O_q(G)$ and $C$ is a cyclic $q'$-group. Write $C=P\times K$ where $P=O_p(G)$ and $K$ is a cyclic group of order prime to $p$ and $q$. Then since $K$ is cyclic, $m_{K,K}\neq 0$ and therefore,
        \begin{align*}
            \Def^G_{G/K}e^G_G=e^{Q\times P}_{Q\times P}\times \Def^K_{K/K}e^K_K=m_{K,K}e^{Q\times P}_{Q\times P}\times e^{K/K}_{K/K}\neq 0\,.
        \end{align*}
        This implies that $\underline{S_{C_q\times C_q,\CC}}(G)=\{0\}$ unless $K=1$. Moreover, since $Q$ is a non-cyclic $q$-group, one has $\beta(Q)\cong C_q\times C_q$. If $N\unlhd Q$ such that $Q/N\cong \beta(Q)$, then $m_{Q,N}\neq 0$. This implies, as above, that $\underline{S_{C_q\times C_q,\CC}}(G)=\{0\}$ unless $Q\cong \beta(Q)\cong C_q\times C_q$. Finally, it is clear that if $G=C_q\times C_q\times P$ for some cyclic $p$-group $P$, then $S_{C_q\times C_q,\CC}(G)$ is non-zero generated by $\overline{e^G_G}$. 
\end{proof}

\begin{proposition}\label{prop conditionqelemabelian}
        Let $G$ be a finite group. Then for any $x\in S_{C_q\times C_q,\CC}(G)$, there exists $x'\in S_{C_q\times C_q,\CC}(G)$ and $\alpha\in I_G$ such that $x=x'+\alpha\cdot x$.
\end{proposition}
\begin{proof}
    The proof is similar to the proof of Proposition~\ref{prop conditionelemabelian}
\end{proof}

As in the previous case, the functor $S_{C_q \times C_q, \CC}$ decomposes as a direct sum of simple $p$-bifree biset functors. In this case, the index set consists of groups of the form $C_q \times C_q \times P$, where $P$ is a cyclic $p$-group.
\begin{theorem}\label{thm elemqabelian}
    One has an isomorphism
    \begin{align*}
        S_{C_q\times C_q,\CC}\cong \bigoplus_{\substack{C_q\times C_q\times P\\ P: \text{cyclic } p\text{-group}}} S_{C_q\times C_q\times P,\CC}^\Delta
    \end{align*}
    of $p$-bifree biset functors.
\end{theorem}
\begin{proof}
    The proof is similar to the proof of Theorem~\ref{thm elemabelian}. By Propositions~\ref{prop restkerqelementary} and \ref{prop conditionqelemabelian} and Theorem~\ref{thm thecondition}, the composition factors of $S_{C_q\times C_q,\CC}$ are exactly the functors $S_{C_q\times C_q\times P,\CC}^\Delta$, where $P$ is a cyclic $p$-group, each with multiplicity one. We will show that each of these functors is a subfunctor.
    
    Let $F$ be a $p$-bifree biset subfunctor of $S:=S_{C_q\times C_q,\CC}$ and let $G$ be a minimal group of $F$. Then we have $\{0\}\neq F(G)\subseteq \underline{S}(G)$ which implies by Proposition~\ref{prop restkerqelementary} that $G$ is isomorphic to $C_q\times C_q\times P$ for some cyclic $p$-group $P$ and that $F(G)$ is one-dimensional generated by $\overline{e^G_G}$. Conversely, if $G\cong C_q\times C_q\times P$ where $P$ is a cyclic $p$-group, then the subfunctor generated by $\underline{S}(G)$ is one-dimensional and has a minimal group $G$. 

    Let $P$ be a cyclic $p$-group and set $G=C_q\times C_q\times P$. Let $F_G$ be the subfunctor of $S$ generated by $\overline{e^G_G}$. We claim that $F_G$ is a simple $p$-bifree biset functor isomorphic to $S_{C_q\times C_q\times P,\CC}^\Delta$. Note that $F_G$ has a minimal group $G$. To show that it is simple, it suffices to show that if $G'=C_q\times C_q\times P'$ is group not isomorphic to $G$, where $P'$ is a cyclic $p$-group, then $F_{G'}\not\subseteq F_G$. One can show that $F_G(G')$ is zero unless $P$ is isomorphic to a subgroup of $P'$. Moreover, in this case, it is one-dimensional generated by $\Ind^{G'}_G e^G_G$, or by $e^{G'}_G$. It follows that if $G'\not\cong G$, then $e^{G'}_{G'}F_G(G')=0$. This implies that $F_{G'}\not\subseteq F_G$. Hence $F_G$ is simple. Since $\overline{e^G_G}$ is invariant under any automorphism of $G$, it follows that $F_G\cong S^\Delta_{G,\CC}$. 
\end{proof}

\section{Appendix}\label{sec appendix}

Let $R$ be a commutative ring and let $\KK$ be a field.

\begin{lemma}\label{lem: restrictionkernelsection}
    Let $F$ be a ($p$-bifree) biset functor over $R$ and $G$ a finite group. Let $\mathrm{Sec}(G)$ be the set of all proper sections of $G$. Then we have
    \[
    \underline{F}(G) = \bigcap_{\substack{H\in \mathrm{Sec}(G)\\ \alpha\in R B^\Delta(H,G)}} \ker (F(\alpha):F(G)\to F(H))\,.
    \]
\end{lemma}
\begin{proof}
    Denote the right hand side of the above equality by $\underline{\underline{F}}(G)$. Clearly $\underline{F}(G)\subseteq\underline{\underline{F}}(G)$. 
    
    For the reverse inclusion, let $T$ be a group of order less than $|G|$ and let $U\le T\times G$. Also let $x\in \underline{\underline{F}}(G)$. It is sufficient to prove that 
    \[
    \Big( \frac{T\times G}{U} \Big)\cdot x = 0. 
    \]
    Write
    \[
    \Big( \frac{T\times G}{U} \Big) = \Ind_P^T\Inf_{P/K}^T\mathrm{iso}\Def^Q_{Q/L}\Res^G_Q
    \]
    with the usual choices of the letters. Thus
    \[
        \Big( \frac{T\times G}{U} \Big)\cdot x = (\Ind_P^T\Inf_{P/K}^T\mathrm{iso}\Def^Q_{Q/L}\Res^G_Q)\cdot x = \Ind_P^T\Inf_{P/K}^T\mathrm{iso}(\Def^Q_{Q/L}\Res^G_Q \cdot x) = 0.
    \]
    Here $\Def^Q_{Q/L}\Res^G_Q \cdot x = 0$ since $x\in\underline{\underline{F}}(G)$.
\end{proof}

The following theorem is stated in \cite{BCK} for Green biset functors. But one can easily check that it is also valid for $p$-bifree biset functors. We include the proof for the sake of self-containment.

\begin{theorem}{\cite[Theorem 2.5]{BCK}}\label{thm thecondition}
Let $F$  be a biset functor over a field $ k $ and let $ G $ be a finite group. Suppose that $ \dim_k F(G) < \infty $, and that for every $ \mathcal E_G $-submodule $ M \subseteq F(G) $, one has
\[
M = (M \cap \underline F(G)) + I_G M.
\]
Then, for every simple $ k[\Out(G)]$-module $ V $, 
the following numbers are equal:
\begin{enumerate}
    \item Multiplicity $[F : S_{G,V}]$ of $S_{G, V}$ as a composition factor of $F$.
    \item Multiplicity of $V$ as a composition factor of the $\mathcal E_G$-module $F(G)$.
    \item Multiplicity of $V$ as a composition factor of the $k[\Out(G)]$-module $\underline F(G)$.
\end{enumerate}
\end{theorem}

\begin{proof} The equality of the first two numbers are well-known. We prove the equality of the last two numbers.
Since $ \dim_k F(G) < \infty $, there exists an $ \mathcal E_G $-composition series
\[
0 = M_0 \subset M_1 \subset \cdots \subset M_{n-1} \subset M_n = F(G)
\]
of $ F(G) $. Set $ K := \underline F(G) $, and consider the induced series
\[
0 = (M_0 \cap K) \subseteq (M_1 \cap K) \subseteq \cdots \subseteq (M_{n-1} \cap K) \subseteq (M_n \cap K) = K
\]
of $ \mathcal E_G $-submodules of $ K $. Let $ V $ be a simple $ k[\Out(G)]$-module, and let $ i \in \{1, \ldots, n\} $. We claim that if $ M_i / M_{i-1} \cong V $, then $ (M_i \cap K)/(M_{i-1} \cap K) \cong V $. This implies that $ [F(G) : V] \leq [K : V] $. But clearly, $ [K : V] \leq [F(G) : V] $, and we obtain equality. 

To prove the claim, observe that
\[
\frac{M_i \cap K}{M_{i-1} \cap K} = \frac{M_i \cap K}{(M_i \cap K) \cap M_{i-1}} \cong \frac{(M_i \cap K) + M_{i-1}}{M_{i-1}} \subseteq \frac{M_i}{M_{i-1}} \cong V.
\]
It therefore suffices to show that the left-hand side is nonzero. Assume by contradiction that $ M_i \cap K = M_{i-1} \cap K $. Then,
\[
M_i = (M_i \cap K) + I_G M_i \subseteq (M_i \cap K) + M_{i-1} = M_{i-1},
\]
contradicting $ M_i / M_{i-1} \cong V \neq 0 $. Here, we used the assumption that $ V $ is annihilated by $ I_G $, and the hypothesis of the theorem.
\end{proof}
\begin{Remark}
Let $F$ be a $p$-bifree biset functor over $\KK$ and let $G$ be a finite group. It is straightforward to prove that the condition in the above theorem is equivalent to any of the following conditions. See \cite[Proposition 2.6]{BCK} for a proof. 
\begin{enumerate}
    \item For every $x\in F(G)$, there exists $x'\in \underline{F}(G)$ and $\alpha\in I_G$ with $x=x'+\alpha\cdot x$. 
    \item For every subfunctor of $M$ of $F$ one has $M(G) = \underline M(G) + \mathcal J M(G)$.
\end{enumerate}
\end{Remark}

\begin{proposition}\label{prop: semisimpledecomposition}
Let $F$ be a semisimple $p$-bifree biset functor. For any finite group $G$, we have
\[
F(G) =\underline{F}(G) \oplus \mathcal JF(G).
\]
In particular, for every $x\in F(G)$, there exists $x'\in\underline{F}(G)$ and $\alpha\in I_G$ satisfying
\begin{align}
    x=x'+\alpha\cdot x\,.
\end{align}
\end{proposition}

\begin{proof}
    Let $n_{H, V}$ denote the multiplicity of the simple $p$-bifree biset functor $S_{H, V}^\Delta$ in $F$. Then evaluating at $G$, we have
    \[
    F(G) \cong \left(\bigoplus_{W} n_{G, W}S_{G, W}^\Delta(G) \right)\oplus \left(\bigoplus_{(H, V): |H|< |G|} n_{H, V}S_{H, V}^\Delta(G)\right)\,.
    \]
    It is clear that the first summand is equal to $\underline{F}(G)$ and that the second summand is equal to $\mathcal JF(G)$. This proves the first assertion. The second assertion follows from the above remark.
\end{proof}



\begin{thebibliography}{00}

\bibitem[BHabil]{BHabil} {\sc R.~Boltje:} 
    Mackey functors and related structures in representation theory and number theory. Habilitation-Thesis, Universität Augsburg 1995.
    
\bibitem[BC18]{BC18} {\sc R.~Boltje, O.~Co\c{s}kun:}
	Fibered biset functors.
    {\sl Adv. Math.} 339 (2018), 540--598.

\bibitem[BCK]{BCK} {\sc R.~Boltje, O.~Co\c{s}kun, \c{C}.~Karag\"{u}zel:}
	The functor of trivial-source modules. preprint.

\bibitem[Bc23]{Bouc2023} {\sc S.~Bouc:}
	Some simple biset functors.
	{\sl Journal of Group Theory} 26 (2023), 1-27.
    
\bibitem[Bc10a]{Bouc2010a} {\sc S.~Bouc:}
	Biset functors for finite groups.
	Lecture Notes in Mathematics, 1990. Springer-Verlag, Berlin, 2010.

\bibitem[Bc07]{Bouc07} {\sc S.~Bouc:}
    The functor of units of Burnside rings for $p$-groups.
    {\sl Comm. Math. Helv.} 82 (2007), 583--615.

\bibitem[Bc06]{Bouc06} {\sc S.~Bouc:}
    The Dade group of a $p$-group.
    {\sl Inv. Math.} 164 (2006), 189--231.

\bibitem[Bc96]{Bouc1996}{\sc S.~Bouc:}
	Foncteurs d'ensembles munis d'une double action.
	{\sl Journal of Algebra} {\bf 183}(3) 1996, 664--736.


\bibitem[BY20]{BoucYilmaz2020} {\sc S.~Bouc, D.~Y{\i}lmaz:}
	Diagonal $p$-permutation functors.
	{\sl Journal of Algebra} {\bf 556} (2020), 1036--1056.

\bibitem[BY22]{BoucYilmaz2022} {\sc S.~Bouc, D.~Y{\i}lmaz:}
    Diagonal p-permutation functors, semisimplicity, and functorial equivalence of blocks. 
    {\sl Adv. Math.}, 411(part A): Paper No. 108799, 54, 2022.

\bibitem[BY25]{BoucYilmaz2025} {\sc S.~Bouc, D.~Y{\i}lmaz:}
	Diagonal $p$-permutation functors in characteristic $p$.
	preprint arXiv:2412.04221.

\bibitem[W10]{Webb2010} {\sc P.~Webb:}
    Stratifications and Mackey functors II: globally defined Mackey functors.
    {\sl Journal of K-Theory} {\bf 6} (2010), no. 1, 99–170.


         




         
        




 

  
	

\bibitem[L18a]{Linckelmann2018a} 
   {\sc M.~Linckelmann:} 
   The block theory of finite group algebras. Vol.~I. 
   London Mathematical Society Student Texts, 91. Cambridge University Press, Cambridge, 2018.
	
	





\end{thebibliography}
\end{document}